\def\draft{n}
\documentclass[12pt]{amsart}
\usepackage[headings]{fullpage}
\usepackage{amssymb,amsmath,bbm}
\usepackage{epic,eepic,epsfig}
\usepackage[all]{xy}
\usepackage{url}
\usepackage[bookmarks=true,%
    colorlinks=true,%
    linkcolor=blue,%
    citecolor=blue,%
    filecolor=blue,%
    menucolor=blue,%
    urlcolor=blue,%
    breaklinks=true]{hyperref}

\newtheorem{theorem}{Theorem}[section]
\theoremstyle{definition}

\newtheorem{lemma}[theorem]{Lemma}

\newtheorem{remark}[theorem]{Remark}

\newtheorem{conjecture}[theorem]{Conjecture}

\def\printname#1{
        \if\draft y
                \smash{\makebox[0pt]{\hspace{-0.5in}
                        \raisebox{8pt}{\tt\tiny #1}}}
        \fi
}

\newcommand{\psdraw}[2]
         {\begin{array}{c} \hspace{-1.3mm}
        \raisebox{-4pt}{\epsfig{figure=draws/#1.eps,width=#2}}
        \hspace{-1.9mm}\end{array}}

\newlength{\standardunitlength}
\catcode`\@=11
\long\def\@makecaption#1#2{%
     \vskip 10pt

\setbox\@tempboxa\hbox{%\ifvoid\tinybox\else\box\tinybox\fi
       \small\sf{\bfcaptionfont #1. }\ignorespaces #2}%
     \ifdim \wd\@tempboxa >\captionwidth {%
         \rightskip=\@captionmargin\leftskip=\@captionmargin
         \unhbox\@tempboxa\par}%
       \else
         \hbox to\hsize{\hfil\box\@tempboxa\hfil}%
     \fi}
\font\bfcaptionfont=cmssbx10 scaled \magstephalf
\newdimen\@captionmargin\@captionmargin=2\parindent
\newdimen\captionwidth\captionwidth=\hsize
\catcode`\@=12

\def\lbl#1{\label{#1}\printname{#1}}

\def\BN{\mathbbm N}
\def\BZ{\mathbbm Z}
\def\BQ{\mathbbm Q}

\def\BC{\mathbbm C}

\def\BW{\mathbbm W}
\def\BWt{\widetilde{\mathbbm W}}

\def\la{\langle}
\def\ra{\rangle}
\def\l{\lambda}

\def\SL{\mathrm{SL}}
\def\longto{\longrightarrow}

\def\pt{\partial}
\def\bl{\boldsymbol{\lambda}}                   %%% {\mathbf{\lambda}}
\def\fsl{\mathfrak{sl}}
\def\fg{\mathfrak{g}}
\def\ga{\gamma}
\def\cxymatrix#1{\xy*[c]\xybox{\xymatrix#1}\endxy}

\begin{document}
\title[The colored HOMFLY polynomial is $q$-holonomic]{The 
colored HOMFLY polynomial is $q$-holonomic}
\author{Stavros Garoufalidis}
\address{School of Mathematics \\
         Georgia Institute of Technology \\
         Atlanta, GA 30332-0160, USA \newline
         {\tt \url{http://www.math.gatech.edu/~stavros }}}
\email{stavros@math.gatech.edu}
\thanks{%
The author was supported in part by grant DMS-0805078 of the US National 
Science Foundation.\bigskip
}

\date{November 26, 2012}%\today }

\begin{abstract}
%In \cite{GL} it was shown that the colored Jones function of a link in
%3-space is $q$-holonomic with respect to the color parameters of a fixed
%simple Lie algebra. In this paper 
We prove that the colored HOMFLY polynomial
of a link, colored by symmetric or exterior powers of the fundamental
representation, is $q$-holonomic with respect to the color parameters.
As a result, we obtain the existence of an $(a,q)$ super-polynomial of all
knots in 3-space. 
%Its specialization to $a=q^N$ and $q=1$ is independent of $N$. 
Our result has implications on the quantization of 
the $\SL(2,\BC)$ character variety of knots using ideal triangulations
or the topological recursion, and motivates 
questions on the web approach to representation theory.
\end{abstract}

\maketitle

\tableofcontents

%\newpage

%%%%%%%%%%%%%%%%%%%%%%%%%%%%%%%%%%%%%%%%%%%%%%%%%%%%%%%%%%%%%%%%%%%%%%%%%%%%
%%%%%%%%%%%%%%%%%%%%%%%%%%%%%%%%%%%%%%%%%%%%%%%%%%%%%%%%%%%%%%%%%%%%%%%%%%%%

\section{Introduction}
\lbl{sec.intro}

\subsection{The colored Jones polynomial}
\lbl{sub.cjones}

In \cite{GL} it was shown that the colored Jones function $J^{\fg}_{L,\l}(q)$
of an oriented link $L$ in 3-space is $q$-holonomic with respect to the color 
parameters $\l$ of a fixed simple Lie algebra $\fg$. 
In particular, if we fix a knot $K$ and a dominant weight
$\l$ of the simple Lie algebra $\fsl_N$, then the sequence $J^{\fsl_N}_{K, n\l}(q)
\in \BZ[q^{\pm 1/2}]$ is $q$-holonomic with respect to $n$. In other words,
the sequence $(J^{\fsl_N}_{K, n\l}(q))_{n \in \BN}$ satisfies a linear 
recursion with coefficients polynomials in $q$ and $q^n$. When $N=2$, the 
minimal order content-free recursion relation is the non-commutative 
$A$-polynomial of $K$ (see \cite{GL,Ga2}) which plays a key role to several
conjectures in Quantum Topology. For a detailed discussion, see \cite{Ga3}
and references therein.

The minimal order content-free recursion for the sequence 
$(J^{\fsl_N}_{K, n\l}(q))_{n \in \BN}$ exists for every $N \geq 2$. There are two
natural questions.
\begin{itemize}
\item
How do the coefficients of these recursions depend on $N$?
\item
What is the specialization of these recursions to $q=1$?
\end{itemize}

The goal of this paper is to answer the first question and pose a natural
conjecture regarding the second question. More precisely, we
prove that the colored HOMFLY polynomial of a link, colored by symmetric 
powers of the fundamental representation of $\fsl_N$, is $q$-holonomic with 
respect to the color parameters, see Theorem \ref{thm.1} below. In particular,
for a knot there is a single recursion for the colored HOMFLY polynomial
colored by the $n$-th symmetric power, such that the coefficients of the
recursion are polynomials in $q$ and $a$. Moreover, specializing $a=q^N$
for fixed $N$ gives a recursion for $(J^{\fsl_N}_{K, n\l_1}(q))$ and further
specializing the coefficients of the recursion to $q=1$ gives a 2-variable
polynomial which is independent of $N$.
As a result, we obtain a rigorously-defined $(a,q)$-variable deformation of 
a two-variable polynomial of a knot, and conjecture that the latter is
the $A$-polynomial of a knot, see Conjecture \ref{conj.1} below.

Our result has implications on the quantization of 
the $\SL(2,\BC)$ character variety of knots using {\em ideal triangulations}
(see \cite{Di1,DG}) or the {\em topological recursion} 
(see \cite{Borot-Eynard}), and motivates questions on the {\em web}
i.e., skein-theory approach to representation theory
(see \cite{Cautis,Ga:Morrison}). We plan to discuss these implications 
in subsequent publications.

\subsection{Super-polynomials in mathematical physics}
\lbl{sub.super}

Mathematical physics has formulated several interesting questions and 
conjectures regarding the structure of the colored HOMFLY polynomial of a knot.
For instance, the LMOV Conjecture of Labastida-Marino-Ooguri-Vafa is 
an integrality statement for the coefficients of the colored HOMFLY 
polynomial, \cite{LMOV}. To formulate enumerative integrality conjectures
concerning counting of BPS states, physicists often use the term 
{\em super-polynomial} in various contexts.
% which is a 
%polynomial deformation of a commutative polynomial
%(often the $A$-polynomial of a knot). 

A super-polynomial (an element of $\BZ[a^{\pm 1},q^{\pm 1}, t^{\pm 1}]$) 
that specializes to the HOMFLY polynomial of a knot when $t=-1$ and its 
$\fsl_N$-Khovanov-Rozansky homology when $a=q^N$ was conjectured to exist
in \cite[Conj.1.2]{Dunfield-Gukov}, motivated by the prior work of 
\cite{Gukov-Vafa}. 

Another super-polynomial (an element of $\BZ[q^{\pm 1},a^{\pm 1},t^{\pm 1}][M]
\la L \ra$, where $LM=qML$) that specializes to a recursion for the 
sequence $(J^{\fsl_N}_{K, n\l}(q))_{n \in \BN}$ when $t=-1$ and $a=q^N$ was
conjectured to exist in \cite{Fu1,Fu2,Gukov-Stosic}. 
A similar conjecture (without the
$t$-variable) was also studied in some cases by \cite{Ito1,Ito2,NRS}.
 
Physics reveals that a super-polynomial is an exciting structure which
ties together perturbative and non-perturbative aspects of quantum knot
theory. For an up-to-date review to this wonderful subject, see the
survey \cite{Gukov:lectures}.

The existence of a super-polynomial has been speculated, but so far not proven.
Theorem \ref{thm.1} settles the existence of the 
super-polynomial of a link colored by the symmetric powers
of the fundamental representation of $\fsl_N$.

In a separate publication we will discuss the $q$-holonomicity of the 
colored HOMFLY polynomial with respect to an arbitrary representation 
of $\fsl_N$.

\subsection{The colored HOMFLY polynomial}
\lbl{sub.cj}

The {\em HOMFLY polynomial} $X_L \in \BQ(a,q)$ of a framed oriented 
link in $S^3$ is uniquely characterized by the skein axioms
%$$
%X\left(\psdraw{L+}{.3in}\right)-X\left(\psdraw{L-}{.3in}\right)
%=(q-q^{-1}) \, X\left(\psdraw{L0}{.3in}\right), \hspace{1cm}
%X\left(\psdraw{curl}{.2in}\right)=a X 
%\left(\psdraw{straight}{.05in}\right)
%$$
$$
X_{\psdraw{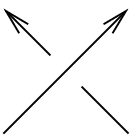}{.2in}}-X_{\psdraw{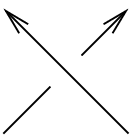}{.2in}}
=(q-q^{-1}) \, X_{\psdraw{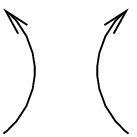}{.2in}} \qquad
X_{\psdraw{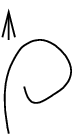}{.1in}}=a X_{\psdraw{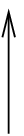}{.02in}} \qquad
X_{\text{unknot}}=\frac{a-a^{-1}}{q-q^{-1}}
$$
%This skein theory specializes when $a=q$ to the skein theory of 
%the Alexander-Conway polynomial, a classical and geometrically 
%well-understood invariant of knots. It is the presence of the framing 
%that makes the framed HOMFLY polynomial interesting, and poorely 
%understood from the point of view of geometry.

The HOMFLY polynomial has an extension to the colored HOMFLY polynomial
$X_{L,\bl}(a,q) \in \BQ(a,q)$, defined for framed oriented links $L$ in $S^3$ 
with 
$r$ ordered components colored by an $r$-tuple of partitions 
$\bl=(\l_1,\dots,\l_r)$. Roughly, the colored HOMFLY polynomial is a sum
of the HOMFLY of parallels of the link, dictated in a universal way by
the color. For a precise and detailed definition, see \cite{Mo1,Mo2}.

\subsection{Recursions and the $q$-Weyl algebra}
\lbl{sub.qweyl}

Recall the notion of a $q$-holonomic sequence introduced by Zeilberger 
\cite{Z90}. We say that a sequence $f_n(q) \in E(q)$ for $n \in \BN$
is $q$-{\em holonomic} (where $E$ is a field of characteristic zero) 
if there exist $d \in \BN$ and $a_j(u,v) \in E[u,v]$ such that for all 
$n \in \BN$ we have:
$$
\sum_{j=0}^d a_j(q,q^n) f_{n+j}(q) = 0
$$
We can write the above recursion in operator form
$$
P f =0, \qquad P=\sum_{j=0}^d a_j(q,M) L^j 
$$
where the operators $M$ and $L$ act on a sequence $(f_n(q))_{n \in \BN}$
by
$$
(Mf_n)(q)=q^n f_n(q), \qquad (Lf_n)(a,q)=f_{n+1}(q)
$$
The operators $M$ and $L$ generate the $q$-Weyl algebra $\BW'=E(q)[M]\la L \ra$
where $LM=qML$. The set $\{P \in \BW' \, | P f=0 \}$ is a left ideal,
non-zero iff and only if $f$ is $q$-holonomic. Although $\BW'$ is not
a principal ideal domain, it was observed in \cite{Ga2} that its 
localization $E(q,M)\la L \ra$ is a principal ideal domain \cite{Ga2}.
If we choose a generator of $\{P \in \BW \, | P f=0 \}$, we can lift it
to a unique content-free element $P_f$ of $\BW'$. Below, we will consider
two versions 
$$
\BWt=\BZ[a,q,M]\la L \ra, \qquad \BW=\BZ[q,M]\la L \ra
$$
of the $q$-Weyl algebra. There is an obvious commutative diagram
\begin{equation}
\lbl{eq.xy}
\cxymatrix{{{\BWt}\ar[rr]^-{a=q^N}\ar[rd]_{a=1,q=1} & & {\BW}\ar[ld]^{q=1}\\
%{V_n(T)}\ar[r]^-{\mathcal R}
&{\BZ[M,L]} &  }}
\end{equation}
Finally, we point out the existence of a multivariable generalization of 
$q$-holonomic sequences $f: \BN \longto E(q)$, see \cite{Sabbah} and also 
\cite{GL}.

\subsection{Our results}
\lbl{sub.results}

\begin{theorem}
\lbl{thm.1}
Fix a framed oriented link $L$ with $r$ ordered components. Then, 
$$
(n_1,\dots,n_r) \mapsto X_{L,(n_1),\dots,(n_r)}(a,q)
\qquad
\text{and}
\qquad
(n_1,\dots,n_r) \mapsto X_{L,(1^{n_1}),\dots,(1^{n_r})}(a,q)
$$ and
are $q$-holonomic functions.
\end{theorem}

\begin{remark}
\lbl{rem.1}
There is an involution $\l \mapsto \l^T$ where $\l^T$ is the transpose of 
$\l$, obtained by interchanging columns and rows. For instance,
if $\l=(4,2,1)$ then $\l^T=(3,2,1,1)$. It is a consequence of rank-level
duality (see \cite[Eqn.4.41]{LMOV}) that 
$X_{K,\l}(a,q)=(-1)^{|\l|}X_{K,\l^T}(a,q^{-1})$, 
where $|\l|$ is the number of boxes of $\l$. Since $(n)^T=(1^n)$, it suffices
to show that $X_{L,(1^{n_1}),\dots,(1^{n_r})}$ is $q$-holonomic with respect to
$(n_1,\dots,n_r)$.
\end{remark}

\begin{remark}
\lbl{rem.cases}
Theorem \ref{thm.1} was previously known for torus links in \cite{Brini},
and for finitely many twist knots (that include the 
$3_1, 4_1, 5_2$ and $6_1$ knots) in \cite{NRS}.
\end{remark}

\begin{theorem}
\lbl{thm.11}
\rm{(a)} For every knot $K$, there exists a unique content-free minimal 
order recursion relation $A_K(a,q,M,L) \in \BWt$ of $X_{K,(n)}(a,q)$.
\newline
\rm{(b)} For every fixed $N \in \BN$, $A_K(q^N,q,M,L) \in \BW$ 
is a recursion of the sequence $(J^{\fsl_N}_{K,n\l_1}(q))$ with respect to $n$.
\newline
\rm{(c)} For every fixed $N \in \BN$, the specialization
$$
A_K(q^N,q,M,L)|_{q=1}=A_K(1,1,M,L)
$$
is independent of $N$.
\end{theorem}

For a fixed natural number $N$, it might be the case that the recursion
$A_K(q^N,q,M,L) \in \BW$ of the sequence $(J^{\fsl_N}_{K,n\l_1}(q))$ is not
of minimal order. If for $N=2$ the above recursion is of minimal order,
then $A_K(q^2,q,M,L) \in \BW$ coincides with 
the non-commutative $A$-polynomial of $K$. In that case, the AJ Conjecture 
(see \cite{Ga2} and also \cite{Ge}) combined with Theorem \ref{thm.11}
imply the following conjecture relating the super-polynomial $A_K(a,q,M,L)$
to the $A$-polynomial $A_K(M,L)$ of $K$, introduced and studied in 
\cite{CCGLS}.

\begin{conjecture}
\lbl{conj.1}
For every knot $K$, we have:
$$
A_K(1,1,M,L)=A_K(M,L) b_K(M) \in \BZ[M,L]
$$
where $A_K(M,L)$ is the $A$-polynomial of $K$ 
and $b_K(M) \in \BZ[M]$.
\end{conjecture}

\subsection{Acknowledgment}
The authors wishes to thank G. Borot, T. Dimofte and S. Gukov for guidance
in the physics literature and T. Le and S. Morrison for many enlightening 
conversations.

%%%%%%%%%%%%%%%%%%%%%%%%%%%%%%%%%%%%%%%%%%%%%%%%%%%%%%%%%%%%%%%%%%%%%%%%%%%%
%%%%%%%%%%%%%%%%%%%%%%%%%%%%%%%%%%%%%%%%%%%%%%%%%%%%%%%%%%%%%%%%%%%%%%%%%%%%

\section{MOY graphs, spiders and webs}
\lbl{sec.MOY}

\subsection{MOY graphs}
\lbl{sub.MOY}

In \cite{MOY}, Murakami-Ohtsuki-Yamada introduced the notion of a MOY
graph which is an enhancement of the HOMFLY skein theory of links
colored by exterior powers of the fundamental representation. MOY graphs 
are similar to the {\em spiders} of Kuperberg \cite{Kuperberg} and the 
{\em webs} of Morrison et al \cite{Morrison,Cautis}.

A {\em MOY graph} is a trivalent planar graph $G$ with oriented edges, possibly
multiple edges and loops with no sinks nor sources. Locally, a MOY graph
is made out of forks and fuses, shown in Figure \ref{fig.ff},
using the terminology of \cite{Morrison,Cautis}.
A {\em coloring} $\ga$ of a MOY graph $G$ is an assignment of a 
natural number to each edge of the graph such that at each fork the {\em flow
condition} is satisfied as shown in Figure \ref{fig.ff}. Here, the color
$i \in \BN$ of an edge corresponds to the $i$-th exterior power of the
fundamental representation.

\begin{figure}[htpb]
$$ 
\psdraw{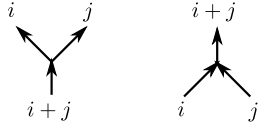}{1.3in} 
$$
\caption{Forks on the left and fuses on the right with a coloring.}\lbl{fig.ff}
\end{figure}

\begin{remark}
\lbl{rem.train}
MOY graphs visually resemble {\em oriented train tracks}. The latter were
introduced by Thurston \cite[Chpt.8]{Thu} and further studied by Penner-Harer
\cite{PH}. We will not make further use of this observation here.
\end{remark}

For $N \in \BN$, \cite[p.328]{MOY} define 
the evaluation $\la G, \ga \ra_N \in \BN[q^{\pm 1}]$ of a colored MOY graph.
In the formulas below, our $q$ equals to $q^2$ in \cite{MOY}. 
%
%If $p$ is a bounded region of a MOY graph $G$, let $\chi_p$ denote the
%corresponding coloring of $G$ {\bf define}. Let 
%$$
%\calB_G=\{\chi_p \, | \, \text{$p$ is a bounded region of $G$}\}
%$$
\begin{lemma}
\lbl{lem.eval}
\rm{(a)} 
%$\calB_G$ is a basis of the monoid of colorings of a MOY graph $G$.
% See also: Baldoni: counting.integer.flows.in.networks.pdf
%
There is a 1-1 correspondence between the set of colorings of a MOY
graph $G$ by integers and $\BZ^r$, where $r$ is the number of bounded regions
of $G$. 
\newline
\rm{(b)} Given $(G,\ga)$ there exists $\la G, \ga \ra \in \BQ(a,q)$ such 
that for every $N \in \BN$ we have
$$
\la G, \ga \ra|_{a=q^N}=\la G, \ga \ra_N
$$
\end{lemma}

\begin{proof}
(a) follows from Kirkhoff's theorem. Fix a connected graph $G$ with oriented
edges, not necessarily planar or trivalent, possibly with loops and 
multiple edges. Then, we have a chain complex 
$$
C: \qquad 0 \to C_1 \overset{\pt}{\to} C_0 \to 0
$$
where $C_1$ and $C_0$ is the free abelian group on the set $E$ of edges and 
$V$ vertices of $G$, and $\pt(e)=v_1-v_0$ if $e$ is an oriented edge 
with head $v_1$ and tail $v_0$. An assignment $\ga: E \to \BZ$ gives
rise to $[\ga]=\sum_{e \in E} \ga(e) e$ and $\pt([\ga])=0$ if and only
if $\ga$ is a coloring. It follows that the set of colorings by integers
is $H_1(C,\BZ)$. Moreover, $H_0(C,\BZ)=\BZ$ since $G$ is connected.
Taking Euler characteristic, it follows that the rank $r$
of $H_1(C,\BZ)$ is given by $-|E|+|V|-1$, where $|X|$ is the number of elements
of $X$. If $G$ is planar, $-|E|+|V|-1=|F|-1$ where $F$ is the set of regions
of $G$. Thus, $r$ is the set of bounded regions.

(b) follows from the definition of the evaluation of a MOY graph, and
the identity
$$
\sum_{i=-\frac{N-1}{2}}^{\frac{N-1}{2}} q^{2j i+k}=
q^k \frac{q^{Nj}-q^{-Nj}}{q^{j}-q^{-j}}=q^k \frac{a^{j}-a^{-j}}{q^{j}-q^{-j}}|_{a=q^N}
$$
\end{proof}

In \cite[p.341]{MOY}, the authors give the following replacement rule of 
a crossing by a $\BZ[q^{\pm 1}]$ linear combination 
of local MOY graphs. For a positive crossing we have 
$$
\Bigg\la \psdraw{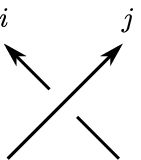}{0.7in} \Bigg\ra=
\begin{cases}
\sum_{k=0}^i (-1)^{k+(j+1)i} q^{i-k} %q^{\frac{i-k}{2}}
\Bigg \la \psdraw{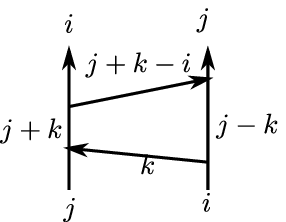}{1.3in} \Bigg \ra
& \qquad \text{for} \qquad i \leq j
\\
\sum_{k=0}^j (-1)^{k+(i+1)j}  q^{j-k} %q^{\frac{j-k}{2}}
\Bigg \la \psdraw{Rfuse1}{1.3in} \Bigg \ra
& \qquad \text{for} \qquad i \geq j
\end{cases}
$$
and for a negative crossing, replace $q$ by $q^{-1}$. It follows that if
$\beta$ is a braid word in the braid group of a fixed number of strings 
and $L$ is the corresponding oriented link obtained by the closure of 
$\beta$, with components colored by $n_i$, then the colored HOMFLY polynomial 
$X_{L,(1^{n_1}),(1^{n_2}),\dots}$ equals to a linear combination the 
evaluations of the corresponding MOY graphs.

Observe that 
\begin{itemize}
\item
the coefficients of the replacement rule are $q$-proper hypergeometric
functions in all color variables. In fact, they are $q$-holonomic monomials
in all color variables,
\item
the replacement rule constructs a finite set of MOY graphs that depend
on $\beta$ and not on the colorings of $L$. The colorings of these graphs 
are linear forms of the colorings of $L$.
\end{itemize}

Since the class of $q$-holonomic functions is closed under summation and
multiplication, Theorem \ref{thm.1} follows from the following theorem
and Remark \ref{rem.1}.

\begin{theorem}
\lbl{thm.2}
For every MOY graph $G$, $\ga \mapsto \la G, \ga\ra$ is $q$-holonomic.
\end{theorem}

Theorem \ref{thm.2} follows from the existence of any $q$-holonomic evaluation 
algorithm. For a detailed discussion of algorithms in quantum topology and 
{\em planar algebras} see \cite[Sec.4]{Bigelow}. Not all of those 
algorithms decrease the complexity. Some do, some keep the complexity 
fixed for a while, and some (like the jelly-fish algorithm of \cite{Bigelow}) 
initially increase the complexity.

An algorithm for evaluating MOY graphs is described in \cite{Jeong-Kim}.
A detailed discussion of the Jeong-Kim algorithm, along with the fixing of some
intermediate steps will be described in a forthcoming publication 
\cite{Ga:Morrison}. Our description of the Jeong-Kim algorithm is implicit in
the recent work of \cite{Cautis} and uses only
the relations (2.3)-(2.8) of \cite{Cautis}.

For the next lemma, we will say that an evaluation algorithm of MOY graphs 
is $q$-holonomic if  
\begin{itemize}
\item[(a)]
the algorithm uses linear relations whose coefficients are 
$q$-holonomic functions of $a=q^N$ and $q$, and 
\item[(b)]
replaces $(G,\ga)$ by graphs $(G',\ga')$ where $G'$ is in a finite set
(determined by $G$) and $\ga'$ are linear forms on $\ga$.
\end{itemize}

Like the proof of Theorem \ref{thm.1}, Theorem \ref{thm.2} is automatically
implied by the following.

\begin{lemma}
\lbl{lem.3}
The evaluation algorithm of \cite{Jeong-Kim} is $q$-holonomic.
\end{lemma}

\begin{proof}
$q$-holonomicity of the evaluation algorithm follows from the fact that
the coefficients in relations (2.3)-(2.8) of \cite{Cautis} are $q$-proper
hypergeometric (in fact, products of $q$-binomials) in linear forms
of the color variables. 
\end{proof}

%%%%%%%%%%%%%%%%%%%%%%%%%%%%%%%%%%%%%%%%%%%%%%%%%%%%%%%%%%%%%%%%%%%%%%%%%%%%
%%%%%%%%%%%%%%%%%%%%%%%%%%%%%%%%%%%%%%%%%%%%%%%%%%%%%%%%%%%%%%%%%%%%%%%%%%%%

\section{Computations}
\lbl{sec.discussion}

There are several papers in the physics literature that discuss the colored
HOMFLY polynomial of a knot, see for example \cite{Ito1,Ito2,NRS,Fu1,Fu2}.
Although the colored HOMFLY polynomial is a well-defined object,  
the formulas for the colored HOMFLY polynomial presented
in the above papers are often void of rigor, although their specializations 
match rigorous computations of the colored Jones polynomial,
and Khovanov Homology. If a formula is written as a multi-dimensional sum
of a $q$-proper hypergeometric summand, then the corresponding function
is $q$-holonomic, and a rigorous computation of a recursion is possible
(see \cite{WZ}) using computer-implemented methods \cite{PWZ}. 
This was precisely done in \cite{NRS}, using some formulas that ought to 
give the colored HOMFLY polynomial of twist knots.

Rigorous formulas for the colored HOMFLY polynomial of a knot are at least
as hard as formulas for the $\fsl_2$-colored Jones polynomial and the latter
are already difficult. For the $1$-parameter family of twist knots, an
iterated 2-dimensional sum for the $\fsl_2$-colored Jones polynomial was 
obtained by Habiro \cite{Ha,Masbaum}. A HOMFLY extension of Habiro's 
formulas appears in \cite{Kawagoe}. Using Kawagoe's formulas for 
the twist knots $3_1$, $4_1$, $5_2$ and $6_1$ and applying the creative
telescoping method of Zeilberger, one can obtain a rigorous computation
of the super-polynomial $A_K(a,q,M,L) \in \BWt$ for the twist knots  
$3_1$, $4_1$, $5_2$ and $6_1$. The results appear in \cite{NRS}.

In addition, a matrix model formulation of the colored HOMFLY polynomial
of torus knots, combined with a topological recursion, provides a rigorous
proof of Theorem \ref{thm.1} for all torus links, see \cite{Brini}.

\bibliographystyle{hamsalpha}
\bibliography{biblio}
\end{document}